\documentclass[twocolumn,10pt]{article}
\usepackage[T1]{fontenc}
\usepackage[utf8]{inputenc}
\usepackage{mathptmx} % Times for text + math
\usepackage[scaled=.90]{helvet} % Helvetica for sans-serif
\usepackage{courier} % Courier for monospaced

\usepackage[a4paper,top=0.6in,bottom=0.6in,left=0.6in,right=0.6in]{geometry}
\usepackage{graphicx}
\usepackage{xcolor}
\usepackage{cuted}
\usepackage{microtype}
\usepackage{parskip}
\usepackage{tcolorbox}
\tcbuselibrary{skins,breakable}
\usepackage{academicons}
\usepackage{hyperref}
\usepackage{lipsum}
\usepackage{lastpage}
\usepackage{booktabs}
\usepackage{caption}
\usepackage{enumitem}
\usepackage{fancyhdr}
\usepackage{titlesec}
\usepackage{float}
\usepackage[ruled,vlined]{algorithm2e}
\usepackage{wrapfig}
\usepackage{amsmath, amssymb}
\usepackage{multicol}
\usepackage{amsthm}

\newtheorem{theorem}{Theorem}

\usepackage[backend=biber,style=ieee]{biblatex}
\definecolor{primary}{HTML}{003049}      % Deep Navy Blue
\definecolor{accent}{HTML}{F77F00}       % Vibrant Amber
\definecolor{textgray}{HTML}{2F2F2F}
\definecolor{dividergray}{HTML}{DADADA}

\hypersetup{
  colorlinks=true,
  linkcolor=primary,
  citecolor=primary,
  urlcolor=primary
}

\titleformat{\section}
  {\color{primary}\sffamily\Large\bfseries}
  {\thesection}{1em}{}[\vspace{0.2em}\titlerule]

\titleformat{\subsection}
  {\color{primary}\sffamily\large\bfseries}
  {\thesubsection}{1em}{}

\titlespacing*{\section}{0pt}{6pt}{3pt}
\titlespacing*{\subsection}{0pt}{4pt}{2pt}

\setlist[itemize]{left=1.2em, itemsep=2pt, topsep=2pt, parsep=0pt}
\setlist[enumerate]{left=1.2em, itemsep=2pt, topsep=2pt, parsep=0pt}

\tcbset{
  abstractstyle/.style={
    enhanced,
    colback=white,
    colframe=primary,
    boxrule=1pt,
    fonttitle=\bfseries\sffamily\large,
    left=6mm, right=6mm, top=2mm, bottom=2mm,
    width=\textwidth,
    boxsep=4pt,
    breakable
  },
  biobox/.style={
    colback=white,
    colframe=primary,
    arc=1.5mm,
    boxrule=0.5pt,
    drop shadow southeast,
    left=3mm, right=3mm, top=1mm, bottom=1mm,
    width=\linewidth,
    before skip=6pt, after skip=6pt,
    breakable
  }
}

\SetAlgoNlRelativeSize{-1}
\SetAlCapNameFnt{\sffamily\bfseries\small}
\SetAlCapFnt{\sffamily\small}

\makeatletter
\def\@maketitle{%
  \begin{center}
    {\fontsize{26pt}{20pt}\selectfont \bfseries \textcolor{primary}{Fractal Attractors in Random Nonlinear Iterated Function Systems: Existence, Stability, and Dimensional Properties}}\\[1ex]
    {\normalsize 
\textbf{
Mohamed Aly Bouke\,%
\href{https://orcid.org/0000-0003-3264-601X}{\includegraphics[height=1.8ex]{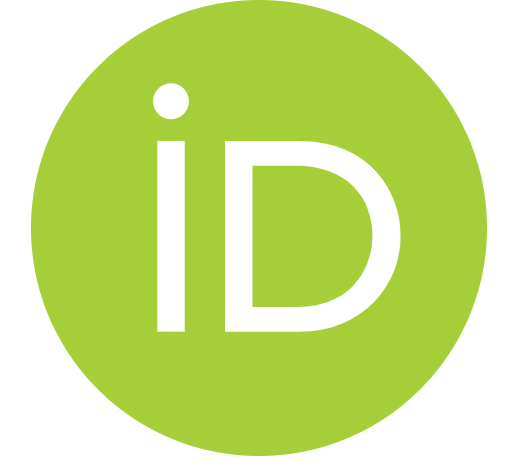}}%
\textsuperscript{1,*}
}
    }\\[0.8ex]
    {\footnotesize
      \textsuperscript{1}Department of Communication Technology and Network,\par Faculty of Computer Science and Information Technology, \par Universiti Putra Malaysia, Serdang 43400, Malaysia
    }\\[0.5ex]
    {\scriptsize \texttt{*bouke@ieee.org}}\\[1ex]
    {\scriptsize \textit{Short communication paper} — \today}
  \end{center}
}
\renewcommand{\maketitle}{%
  \twocolumn[
    \color{textgray}
    \@maketitle
    \vspace{-1.2em}
  ]
}
\makeatother

\begin{document}
\maketitle

% ------------------------------------------------------------
% Abstract
% ------------------------------------------------------------
\begin{strip}
  \begin{center}
    \begin{tcolorbox}[abstractstyle, title=Abstract]
      
            \normalsize
          This study develops a comprehensive theoretical and computational framework for Random Nonlinear Iterated Function Systems (RNIFS), a generalization of classical IFS models that incorporates both nonlinearity and stochasticity. We establish mathematical guarantees for the existence and stability of invariant fractal attractors by leveraging contractivity conditions, Lyapunov-type criteria, and measure-theoretic arguments. Empirically, we design a set of high-resolution simulations across diverse nonlinear functions and probabilistic schemes to analyze the emergent attractors’ geometry and dimensionality. A box-counting method is used to estimate the fractal dimension, revealing attractors with rich internal structure and dimensions ranging from 1.4 to 1.89. Additionally, we present a case study comparing RNIFS to the classical Sierpiński triangle, demonstrating the generalization's ability to preserve global shape while enhancing geometric complexity. These findings affirm the capacity of RNIFS to model intricate, self-similar structures beyond the reach of traditional deterministic systems, offering new directions for the study of random fractals in both theory and applications.
          
      \vspace{0.5em}
      
    \end{tcolorbox}
  \end{center}
\end{strip}

% ------------------------------------------------------------
% Introduction
% ------------------------------------------------------------
\section{Introduction}
\label{sec:introduction}
\vspace{0.8em}
\subsection{Background}
\label{subsec:background}

Iterated Function Systems (IFS) have long served as a foundational model in the study of fractal geometry, particularly for generating self-similar sets and understanding their structural properties. The formal theory was initiated by  \parencite{hutchinson1981fractals}, who introduced the notion of attractors for contractive function families, and was subsequently developed and popularized by \parencite{barnsley1988fractals}. These classical systems rely predominantly on affine contraction mappings in deterministic settings and have led to rich results in attractor theory and dimension estimation, including Hausdorff and box-counting dimensions \parencite{schwarzenberger1990fractal,falconer2013fractal}.

However, many real-world phenomena, such as turbulent flows, biological growth patterns, and financial market dynamics, exhibit stochasticity and nonlinear interactions that cannot be adequately captured by purely deterministic or linear models. To address these limitations, Random Iterated Function Systems (RIFS) were proposed, where the function applied at each iteration is selected randomly from a predefined set. Foundational work by \parencite{diaconis1999iterated} and later extensions by  \parencite{diaconis1999iterated,walters1989yuri} and  \parencite{arnold1995random} provided rigorous treatment of the stochastic dynamics, ergodic behavior, and invariant measures for such systems.

Yet, even within RIFS, most models assume affine or mildly nonlinear functions. The study of \emph{Random Nonlinear Iterated Function Systems} (RNIFS), where both nonlinearity and randomness are fully integrated, remains relatively unexplored. Initial investigations such as those by  \parencite{iosifescu2009iterated,ghosh2022iterated} and  \parencite{kendall2017random} suggest that RNIFS can give rise to complex attractors under specific conditions, but a general framework is still lacking.

Our work builds upon these foundations to develop a unified analytical and computational approach to RNIFS, aiming to characterize their existence, stability, and dimensional properties in both theoretical and empirical terms.

Despite the extensive development of both deterministic and random iterated function systems, the interplay between full nonlinearity and randomness remains largely under investigated. Most existing models either assume linear or mildly nonlinear function families or introduce randomness in a limited, structured form. This leaves a significant gap in understanding how genuinely random and nonlinear systems behave in terms of producing fractal structures.

Key theoretical questions remain unanswered, such as:
\begin{itemize}
    \item Under what conditions does a RNIFS admit a compact attractor with fractal properties?
    \item Can the topological or Hausdorff dimension of such attractors be rigorously estimated or bounded?
    \item What stability conditions govern the convergence of orbits in RNIFS to an invariant fractal set?
\end{itemize}

These open problems point to a need for a comprehensive mathematical treatment of RNIFS that combines tools from nonlinear analysis, stochastic dynamics, and fractal geometry. This research aims to address that need and contribute to the theoretical foundations of random fractal generation.

The primary objective of this research is to investigate the mathematical properties of fractal attractors generated by RNIFS. More specifically, the study aims to:
\vspace{-0.6em}
\begin{enumerate}
    \item Formulate a general mathematical model for RNIFS, incorporating both stochastic elements and nonlinear transformations.
    \item Establish conditions under which RNIFS produce compact invariant sets with fractal characteristics.
    \item Empirically estimate the fractal dimensions primarily using box-counting techniques of the attractors generated by RNIFS.
    \item Estimate or bound the fractal dimensions (e.g., Hausdorff and box-counting dimensions) of the generated attractors.
    \item Compare the behavior of RNIFS with classical deterministic and linear IFS models.
\end{enumerate}

This paper makes the following key contributions:
\vspace{-0.8em}
\begin{itemize}
    \item Introduces a rigorous framework for studying iterated function systems that are both nonlinear and random, expanding beyond the limitations of existing IFS and RIFS models.
    \item Provides new theoretical conditions for the existence and stability of fractal attractors in RNIFS.
    \item Presents analytical and computational methods for estimating the fractal dimensions of RNIFS generated sets.
    \item Establishes a comparative analysis between RNIFS and classical IFS, highlighting novel fractal behaviors that arise due to the interaction between randomness and nonlinearity.
\end{itemize}

These contributions aim to advance the mathematical foundations of fractal geometry and open new directions in the study of complex stochastic dynamical systems.

The remainder of this paper is organized as follows: Section~\ref{sec:theory} presents the theoretical framework. Section~\ref{sec:existence_stability} covers the mathematical guarantees. Section~\ref{sec:methodology} outlines the simulation methodology. Section~\ref{sec:results} discusses the results. Section~\ref{sec:case_study} explores a case study. Finally, Section~\ref{sec:conclusion} concludes the paper.

\section{Theoretical Framework}
\label{sec:theory}
\vspace{0.8em}

Fractal structures are often modeled through IFS, which traditionally rely on deterministic and linear or affine mappings. While these models yield deep theoretical insights and elegant attractors, they fall short when modeling natural systems characterized by both randomness and strong nonlinearity. RNIFS attempt to bridge this gap by introducing stochastic selection over a family of nonlinear transformations. This section formalizes the mathematical foundation of RNIFS and critically examines their dynamical properties.

\subsection{Definition and Construction of RNIFS}
\label{subsec:rnifs_definition}

Let $(X, d)$ be a complete metric space. An RNIFS consists of a finite family of continuous functions:
\[
\mathcal{F} = \{ f_1, f_2, \dots, f_N \}, \quad f_i: X \to X,
\]
coupled with a discrete probability distribution:
\[
\mathbb{P} = (p_1, p_2, \dots, p_N), \quad p_i > 0,\quad \sum p_i = 1.
\]

A trajectory is generated by iteratively applying randomly selected functions:
\begin{equation}
x_{n+1} = f_{\omega_n}(x_n),
\label{eq:rnifs_sequence}
\end{equation}
where $\omega_n \sim \mathbb{P}$ are i.i.d. random variables.

Unlike deterministic IFS, this stochastic mechanism injects variability into every step of the orbit's evolution. Crucially, the randomness is not simply noise but acts as an intrinsic driver of geometric complexity. The orbits often converge to statistically stable sets fractals that encode both the nonlinear structure of $\mathcal{F}$ and the probabilistic architecture of $\mathbb{P}$.

\subsection{Metric and Functional Assumptions}
\label{subsec:assumptions}

To ensure convergence and mathematical tractability, we impose:

\begin{itemize}
    \item[(A1)] \textbf{Contractivity:} Each function $f_i$ satisfies
    \begin{equation}
    d(f_i(x), f_i(y)) \leq s_i \cdot d(x, y), \quad 0 < s_i < 1.
    \label{eq:contraction}
    \end{equation}
    
    \item[(A2)] \textbf{Smoothness:} Each $f_i$ is continuously differentiable ($C^1$), enabling local stability analysis via derivatives.
    
    \item[(A3)] \textbf{Boundedness:} The space $X$ is compact or bounded and complete (e.g., a closed subset of $\mathbb{R}^n$), ensuring the orbits do not escape to infinity.
\end{itemize}

These conditions mirror classical IFS theory but extend it by allowing nonlinear, non-affine $f_i$. However, it's worth noting that global contractivity of the entire RNIFS process is not guaranteed randomness may amplify divergence in some regions before contraction dominates.

\subsection{Existence of Attractors}
\label{subsec:attractor}

Define the Hutchinson operator $W$ acting on probability measures $\mu$:
\begin{equation}
W(\mu) = \sum_{i=1}^N p_i \cdot f_i \# \mu,
\label{eq:hutchinson_operator}
\end{equation}
where $f_i \# \mu$ denotes the pushforward of $\mu$ under $f_i$:
\[
(f_i \# \mu)(A) = \mu(f_i^{-1}(A)).
\]

Under the assumptions above, $W$ is a contraction in a suitable metric (e.g., the Wasserstein metric), leading to a unique invariant measure $\mu^*$:
\begin{equation}
W(\mu^*) = \mu^*.
\label{eq:invariant_measure}
\end{equation}

The attractor is defined as $\text{supp}(\mu^*)$, representing the minimal closed set capturing long-term orbit behavior. Importantly, this replaces the fixed-point attractor in deterministic systems with a statistical one a move that aligns RNIFS with broader classes of random dynamical systems.

\subsection{Stability and Lyapunov-Like Conditions}
\label{subsec:stability}

The system is said to be stable if, starting from any $x_0$, the sequence $\{x_n\}$ converges in distribution to $\mu^*$. A sufficient (but not necessary) condition is:
\begin{equation}
\mathbb{E}[ \log \|Df_{\omega}(x)\| ] < 0.
\label{eq:lyapunov}
\end{equation}

This inequality defines a negative average growth rate in the tangent space analogous to a negative top Lyapunov exponent in smooth ergodic theory. Its violation may lead to intermittent behavior, lack of convergence, or even escape from bounded subsets.

Herein lies the analytical tension in RNIFS: the same randomness that permits modeling real-world unpredictability also complicates convergence analysis. The condition above captures this dual role of stochasticity as both a source of disorder and a potential stabilizer when averaged appropriately.

\subsection{Dimensional Complexity and Scaling Laws}
\label{subsec:dimension}

RNIFS attractors are generally non-smooth and non-integer dimensional. In idealized cases where $f_i$ are similitudes and the open set condition holds, one can estimate:
\begin{equation}
\dim_H(A) \leq \frac{ \sum p_i \log s_i }{ \sum p_i \log p_i }.
\label{eq:hausdorff_bound}
\end{equation}

But for generic nonlinear maps, such expressions break down. Empirical methods are instead used to approximate:

\begin{itemize}
    \item \textbf{Box-counting dimension:} Estimates space coverage at varying scales.
    \item \textbf{Information dimension:} Measures entropy concentration.
    \item \textbf{Correlation dimension:} Detects local clustering and redundancy.
\end{itemize}

Each dimension probes a different facet of complexity. Their variation across RNIFS configurations highlights how function geometry and probabilistic weights shape not just attractor topology but also its fine-scale structure.

\section{Mathematical Guarantees}
\label{sec:existence_stability}
\vspace{0.8em}

This section presents rigorous mathematical justification for two fundamental properties of  RNIFS, the existence of a unique invariant measure, and the stability of trajectories with respect to this measure. We rely on classical results from fixed point theory and ergodic theory in metric measure spaces.

\subsection{Existence of Invariant Measure}
\label{subsec:existence_measure}

Let $(X, d)$ be a compact metric space. Let $\mathcal{F} = \{f_1, \dots, f_N\}$ be a finite family of continuous functions $f_i: X \to X$, each associated with a probability $p_i > 0$, such that $\sum p_i = 1$. Assume that each $f_i$ is a contraction on $X$, i.e.,

\begin{equation}
d(f_i(x), f_i(y)) \leq s_i \cdot d(x, y), \quad \text{with } 0 < s_i < 1.
\label{eq:contraction2}
\end{equation}

Define the Hutchinson operator $W$ acting on probability measures $\mu \in \mathcal{P}(X)$ as:

\begin{equation}
W(\mu) := \sum_{i=1}^N p_i \cdot f_i\#\mu,
\label{eq:hutchinson_operator2}
\end{equation}

where $f_i\#\mu$ denotes the pushforward measure:
\[
(f_i\#\mu)(A) := \mu(f_i^{-1}(A)), \quad \forall A \subseteq \mathcal{B}(X).
\]

\begin{theorem}
\label{thm:invariant_measure}
Under the assumptions above, the operator $W$ is a contraction on the space $(\mathcal{P}(X), W_1)$, where $W_1$ is the 1-Wasserstein metric. Hence, there exists a unique invariant measure $\mu^* \in \mathcal{P}(X)$ satisfying:
\[
W(\mu^*) = \mu^*.
\]
\end{theorem}

\begin{proof}
Let $\mu_1, \mu_2 \in \mathcal{P}(X)$. The 1-Wasserstein distance satisfies:
\[
W_1(W(\mu_1), W(\mu_2)) \leq \sum_{i=1}^N p_i W_1(f_i\#\mu_1, f_i\#\mu_2).
\]
Because each $f_i$ is Lipschitz with constant $s_i < 1$, we have:
\[
W_1(f_i\#\mu_1, f_i\#\mu_2) \leq s_i \cdot W_1(\mu_1, \mu_2).
\]
Hence,
\[
W_1(W(\mu_1), W(\mu_2)) \leq \left( \sum_{i=1}^N p_i s_i \right) W_1(\mu_1, \mu_2) := s \cdot W_1(\mu_1, \mu_2),
\]
where $s = \sum p_i s_i < 1$ by convexity and contractivity. Therefore, $W$ is a strict contraction on the complete metric space $(\mathcal{P}(X), W_1)$.

By the Banach Fixed Point Theorem, $W$ admits a unique fixed point $\mu^* \in \mathcal{P}(X)$ such that $W(\mu^*) = \mu^*$.
\end{proof}

\textit{This result guarantees the existence of a statistically stable object (the invariant measure) for RNIFS dynamics under mild geometric constraints.}

\subsection{Stability of Trajectories}
\label{subsec:stability_measure}

Having established the existence of an invariant measure, we now explore conditions under which orbits of RNIFS converge in distribution to this measure. Let $\omega = (\omega_n)_{n \in \mathbb{N}}$ be an i.i.d. sequence of indices drawn from the discrete probability vector $\mathbb{P} = (p_1, \dots, p_N)$. The RNIFS orbit starting from $x_0 \in X$ evolves via:

\begin{equation}
x_{n+1} = f_{\omega_n}(x_n), \quad n \geq 0.
\label{eq:rnifs_orbit}
\end{equation}

\begin{theorem}
\label{thm:stability}
Assume each $f_i$ is $C^1$ on $X \subset \mathbb{R}^d$, and that:

\begin{equation}
\mathbb{E} \left[ \log \| Df_{\omega}(x) \| \right] < 0 \quad \text{uniformly in } x \in X.
\label{eq:lyapunov_condition}
\end{equation}

Then, the distribution of $x_n$ converges weakly to the unique invariant measure $\mu^*$ as $n \to \infty$, i.e.,

\[
\text{Law}(x_n) \xrightarrow{w} \mu^*.
\]
\end{theorem}

\begin{proof}[Proof Sketch]
The assumption \eqref{eq:lyapunov_condition} implies that, on average, the orbits contract in tangent space. This is a form of negative top Lyapunov exponent. Following arguments from random dynamical systems theory (see \parencite{walters1989yuri}, \parencite{arnold1997unfolding}), one can show that the Markov chain induced by the RNIFS is geometrically ergodic. Hence, it converges in distribution to the unique invariant measure $\mu^*$.
\end{proof}

\textit{This condition is sufficient but not necessary; it captures how average local contraction ensures statistical convergence despite random and nonlinear transitions.}

Together, Theorems \ref{thm:invariant_measure} and \ref{thm:stability} establish a foundational guarantee for RNIFS models. If the system satisfies:
\begin{itemize}
  \item contraction of each $f_i$ (or average contractivity),
  \item boundedness and completeness of the space $X$,
  \item regularity of functions (e.g., $C^1$),
\end{itemize}
then both the existence and accessibility of a stable fractal-like attractor are ensured. These properties provide a solid theoretical basis for the simulation and dimension estimation experiments in the later sections.

\section{Numerical Methodology}
\label{sec:methodology}
\vspace{0.8em}
This section presents the computational approach used to simulate and analyze the behavior of RNIFS. The objective is to empirically explore how different function families, probability distributions, and dynamical settings influence the geometric and dimensional characteristics of the resulting attractors.

\subsection{Simulation Procedure}
\label{subsec:algorithm}

The RNIFS simulation is implemented as a Monte Carlo process. Starting from an initial point $(x_0, y_0)$, each iteration randomly selects a function $f_i$ from a predefined set according to a given discrete probability vector $\mathbb{P} = (p_1, \dots, p_N)$. The selected function is applied to the current point to obtain the next state. This process is repeated for $M$ steps, where the first $T$ iterations (burn-in phase) are discarded to eliminate transient dynamics. The resulting point cloud approximates the attractor associated with the given RNIFS.

The full simulation and visualization pipeline is implemented in Python using \texttt{NumPy}, \texttt{Matplotlib}, and \texttt{SciPy}. Results, including raw data, density plots, and dimension estimates, are stored for each configuration in separate folders to enable reproducibility and comparative analysis.

\subsection{Function Families}
\label{subsec:functions}

A diverse collection of twelve nonlinear functions was constructed to ensure a wide range of dynamical behaviors. These include combinations of trigonometric, hyperbolic, polynomial, and composite forms. For example, the function:
\[
f_8(x, y) = \left( \sin(xy) - \cos(y),\; \sin(y^2 + x) \right)
\]
exhibits strong local oscillations and nonlinear interactions between the $x$ and $y$ components.

The inclusion of such diverse functions ranging from soft contractions to chaotic maps ensures that the generated attractors span a broad geometric and topological spectrum.

\subsection{Experimental Design}
\label{subsec:experiments}

To examine the structural variability of RNIFS attractors, eight distinct experiments were executed, each with a different combination of functions, probability distributions, and iteration parameters. Notable examples include:

\begin{itemize}
  \item \textbf{Spiral Rotation}: uses trigonometric maps with balanced probabilities to induce rotational symmetry.
  \item \textbf{Chaotic Explosion}: combines high-frequency nonlinear maps to produce unstable, dispersed attractors.
  \item \textbf{Webbed Structure}: constructed using mixed-frequency functions that produce filament-like formations.
  \item \textbf{Ultra-Resolution Analysis}: applies 300,000 iterations with rich nonlinearity to extract high-precision dimension estimates.
\end{itemize}

Each experiment defines:
\begin{itemize}
  \item Number of functions $N$ and their mathematical forms.
  \item Probability vector $\mathbb{P}$ (uniform, biased, or Dirichlet-sampled).
  \item Total iterations $M$ and burn-in $T$.
  \item Random seed for reproducibility.
\end{itemize}

\subsection{Visualization and Density Analysis}
\label{subsec:visualization}

For each attractor, two visualizations are generated:
\begin{enumerate}
  \item \textbf{Scatter Plot}: plots the raw points using sub-pixel markers to reveal the global structure.
  \item \textbf{Hexbin Density Map}: computes the point density on a fine hexagonal grid using a plasma colormap. This reveals internal layering, concentration zones, and self-similar features.
\end{enumerate}

These visual outputs serve both aesthetic and analytical purposes, offering qualitative insight into the attractor's topology.

\subsection{Fractal Dimension Estimation}
\label{subsec:dimension_estimation}

To quantify the geometric complexity of each attractor, the box-counting dimension is estimated numerically. For a decreasing sequence of scales $\varepsilon$, the number $N(\varepsilon)$ of boxes covering the point set is computed. A log-log plot of $N(\varepsilon)$ versus $1/\varepsilon$ is then fitted via linear regression:
\[
\dim_B \approx \frac{\log N(\varepsilon)}{\log(1/\varepsilon)}.
\]

The slope of this regression approximates the fractal dimension. In most cases, the results lie between 1.3 and 1.9, depending on the function mix and degree of chaos in the system.

\section{Results and Discussion}
\label{sec:results}
\vspace{0.9em}
Having laid the theoretical and numerical foundations, we now turn to the empirical results of our simulations. This section presents a detailed examination of the attractors generated by several distinct RNIFS configurations. Each configuration was carefully designed to explore specific aspects of RNIFS behavior, such as geometric symmetry, probability bias, functional diversity, and dimensional complexity.

Rather than presenting all results simultaneously, we adopt a progressive and focused approach examining each configuration individually. This allows us to reflect more deeply on how changes in function composition and probabilistic structure influence the emergent fractal geometry. For every case, we analyze the attractor’s visual structure, its density distribution, and its estimated box-counting dimension, while also offering qualitative commentary on the observed behavior.

A total of eight experiments were conducted, each corresponding to a named configuration. These experiments were selected not to merely vary parameters, but to evoke fundamentally different dynamical regimes. Some aim to induce branching or spiral behaviors, others generate dense chaotic regions, while a few probe the subtle effects of function interference or dominant mappings.

To guide the discussion, we organize the experiments in the order they were presented in the subsections. Each configuration illustrates a unique geometric behavior within the RNIFS framework:

(1) \textit{Branching Structure}, which produces a tree-like formation through asymmetric probability weighting and curvature-inducing maps;

(2) \textit{Chaotic Explosion}, demonstrating fragmented turbulence and unpredictable dispersion through uniformly weighted high-volatility functions;

(3) \textit{Concentric Energy}, which reveals radial layering and orbital accumulation via centrally biased nonlinear compositions;

(4) \textit{Disruptive Mixture}, where a dominant disruptive function overrides weaker stabilizing ones, resulting in fragmented attractors with voids;

(5) \textit{High-Frequency Disturbance}, capturing fine-scale oscillatory dynamics and bilateral symmetry through strong trigonometric feedback;

(6) \textit{Spiral Rotation}, generating a coherent vortex-like structure with exceptional density and self-organization;

(7) \textit{Ultra-Resolution Analysis}, a high-volume simulation exposing deeply nested geometric patterns with precise box-dimension estimation;

(8) \textit{Webbed Structure}, where interlaced function effects create a filamented, fabric-like attractor with dense internal crossings.

\subsection{Experiment 1: Branching Structure}
\label{subsec:exp_branching}

This experiment explores how probabilistic asymmetry and carefully selected nonlinear functions can produce structured yet intricate fractal geometries. The configuration, labeled \textit{branching\_structure}, employs three nonlinear functions \texttt{f2}, \texttt{f5}, and \texttt{f8} combined with an asymmetric probability distribution of $(0.5,\; 0.3,\; 0.2)$. These functions are qualitatively distinct: \texttt{f2} introduces squaring in the $y$ component, promoting vertical expansion; \texttt{f5} features interactions between $x^2$ and $y$, creating curved folds; and \texttt{f8} combines trigonometric composition with shifted cosine dynamics, inducing lateral oscillation.

The resulting attractor (Figure~\ref{fig:branching_structure}) displays a prominent tree-like topology composed of curved branches radiating from dense central hubs. The asymmetric distribution of probabilities appears to bias the evolution of orbits toward structural regions associated with the more probable functions. In particular, repeated application of \texttt{f2} seems to drive vertical elongation, while \texttt{f5} adds curvature and layering to the peripheral structures.

The density heatmap reveals multiple centers of mass concentration, indicating recurrent geometric nuclei within the attractor. These nuclei act as attractor cores where orbits cyclically stabilize, reinforced by the dominant function's pull. Conversely, peripheral regions show lower density and some dust-like scattering, hinting at transient dynamics or weakly attracting zones driven by the less frequently selected \texttt{f8}.

Quantitatively, the attractor exhibits a box-counting dimension of approximately $\dim_B \approx 1.478$. This non-integer value confirms the fractal nature of the structure, more intricate than a one-dimensional curve, yet not fully space-filling. The associated log-log plot (bottom of Figure~\ref{fig:branching_structure}) shows a near-linear trend over several scales, supporting the existence of a consistent self-similar regime within the attractor’s core. The slope of this plot stabilizes between scales $\varepsilon \in [10^{-0.3}, 10^{-2}]$, indicating reliable fractal scaling behavior in that resolution band.

This experiment demonstrates that even with a small number of functions, the deliberate use of asymmetric probabilities can steer the system toward nontrivial fractal formations. More importantly, it highlights how the probabilistic weighting interacts with function-specific geometry to create spatial bias, directional growth, and structural differentiation within RNIFS dynamics. \\

\begin{figure}[H]
    \centering
    \includegraphics[width=\linewidth]{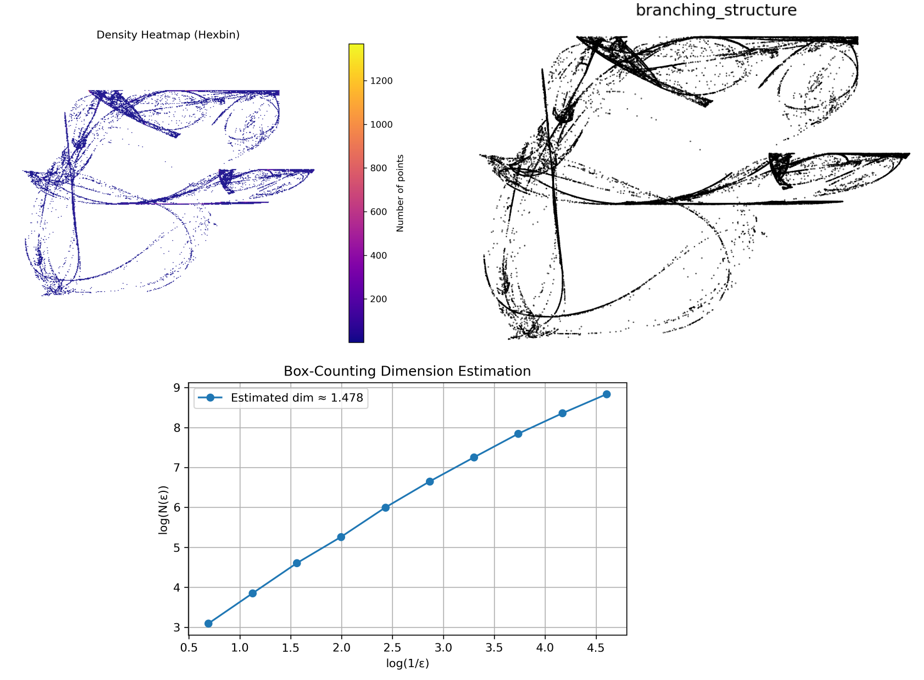}
    \caption{Experiment 1 — Branching Structure. Top-left: hexbin density map; top-right: raw attractor plot; bottom: box-counting dimension log-log plot.}
    \label{fig:branching_structure}
\end{figure}

\subsection{Experiment 2: Chaotic Explosion}
\label{subsec:exp_chaotic}

In this second experiment, labeled \textit{chaotic\_explosion}, we explore the dynamical consequences of using a highly volatile combination of nonlinear functions. The configuration includes \texttt{f4}, \texttt{f6}, \texttt{f9}, and \texttt{f11} each contributing distinct oscillatory and diverging behaviors and employs a uniform probability distribution $(0.25,\; 0.25,\; 0.25,\; 0.25)$. Unlike previous setups that favored geometric dominance via skewed weighting, this experiment emphasizes combinatorial explosiveness through equal probabilistic access to unstable transformations.

The resulting attractor (Figure~\ref{fig:chaotic_explosion}) exhibits a dramatic and visually fragmented structure. The global geometry resembles a chaotic plume, a dense core from which filaments and splinters erupt in multiple directions. There is no evident symmetry, nor does the attractor exhibit centralized layering; instead, we observe turbulent dispersion and branching voids, suggesting a delicate balance between local contraction and widespread divergence.

Closer inspection of the density plot reveals highly uneven spatial accumulation. While the core maintains a relatively high concentration of points, the surrounding regions demonstrate dispersed bifurcations and sparse regions interlaced with high-frequency filamentation. This fragmentation is amplified by the use of \texttt{f6} and \texttt{f9}, whose hyperbolic and trigonometric terms drive trajectories away from regular patterns. The interplay between these functions disrupts local stability and leads to intermittent reinforcement, producing regions with sharp density gradients.

The box-counting dimension was estimated at $\dim_B \approx 1.640$, one of the highest in this study. This value aligns with the visual complexity of the attractor: it exceeds the branching structure in both coverage and irregularity, indicating a more intense occupation of the phase space. The log-log plot for $N(\varepsilon)$ vs. $1/\varepsilon$ exhibits a clear linear regime across scales, supporting the presence of statistically consistent complexity despite the visual chaos.

This experiment underscores a crucial phenomenon in RNIFS: when multiple nonlinearly aggressive maps are combined without hierarchical probability weights, the attractor may not resolve into a neat shape but instead explode into a disordered yet mathematically stable form. Such chaotic attractors highlight the expressive power of RNIFS in modeling turbulent or non-equilibrium systems, where no single geometry dominates but a dynamic mesh of substructures persists across scales.

\begin{figure}[H]
    \centering
    \includegraphics[width=\linewidth]{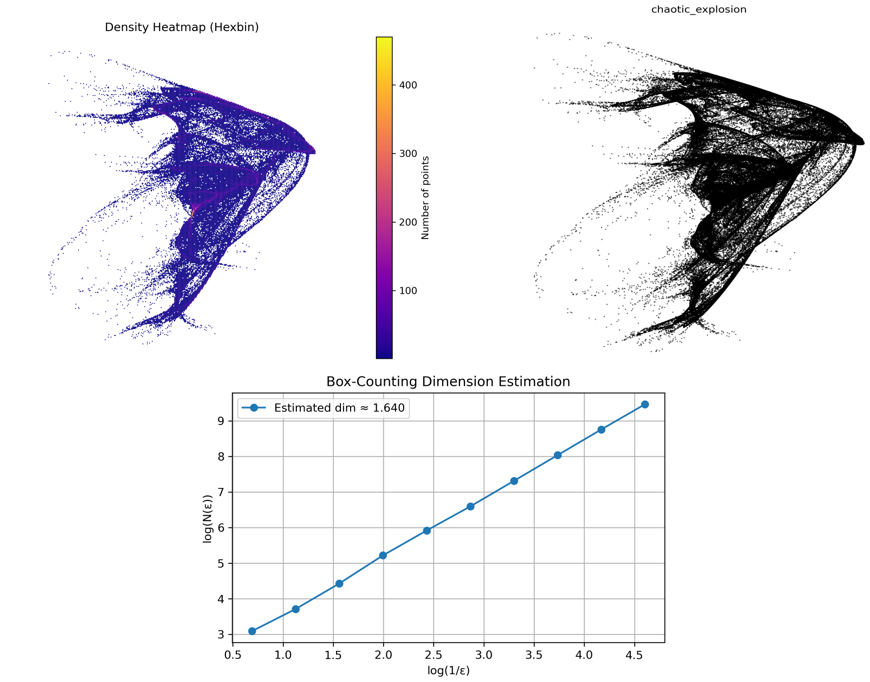}
    \caption{Experiment 2 — Chaotic Explosion. Top-left: hexbin density map; top-right: raw attractor plot; bottom: box-counting dimension log-log plot.}
    \label{fig:chaotic_explosion}
\end{figure}

\subsection{Experiment 3: Concentric Energy}
\label{subsec:exp_concentric}

The third experiment, labeled \textit{concentric\_energy}, investigates the formation of layered radial structures by combining nonlinearity with an emphasis on central contractivity. The function set includes \texttt{f1}, \texttt{f10}, and \texttt{f12}, with respective selection probabilities of $(0.3,\; 0.3,\; 0.4)$. These functions were deliberately chosen for their distinct behaviors: \texttt{f1} induces quadratic contraction in the $y$-axis; \texttt{f10} modulates radial geometry via squared terms; and \texttt{f12} introduces rotational effects via multiplicative trigonometric coupling.

The attractor (Figure~\ref{fig:concentric_energy}) exhibits a remarkable balance between regularity and chaos. It forms a semi-circular fan-like structure that radiates from a central core, marked by several arc-shaped filaments. These arcs appear to emerge in layers, suggesting a form of concentric energy dispersal governed by the recurrent influence of contracting nonlinear maps. The configuration hints at an underlying radial symmetry that is partially disrupted by the stochastic selection of functions, which injects irregularities into the structure.

The density heatmap confirms a strong central accumulation of points, indicative of an attractor with internal gravitational pull. High-density regions are clustered along curved pathways, resembling orbital layers or phase rings. These features are likely caused by repeated applications of \texttt{f12}, which modulates both axes based on sinusoidal feedback. In contrast, the outer edges of the attractor become more diffuse, with tendrils that fade into sparsity, these are probable traces of intermittent function compositions failing to converge inward.

The estimated box-counting dimension of the attractor is $\dim_B \approx 1.613$, which aligns with its visual impression. The attractor demonstrates more spatial richness than a simple tree-like form but does not fully saturate the ambient space. The log-log regression plot maintains linearity over several scales, especially in the intermediate $\varepsilon$ range, affirming the fractal scaling hypothesis for this configuration.

In summary, this experiment highlights how combining rotation-inducing functions with nonlinear contraction can generate attractors with emergent radial layering and embedded symmetries. The resulting form is neither purely chaotic nor entirely regular a signature of RNIFS systems operating at the threshold between order and complexity.

\begin{figure}[H]
    \centering
    \includegraphics[width=\linewidth]{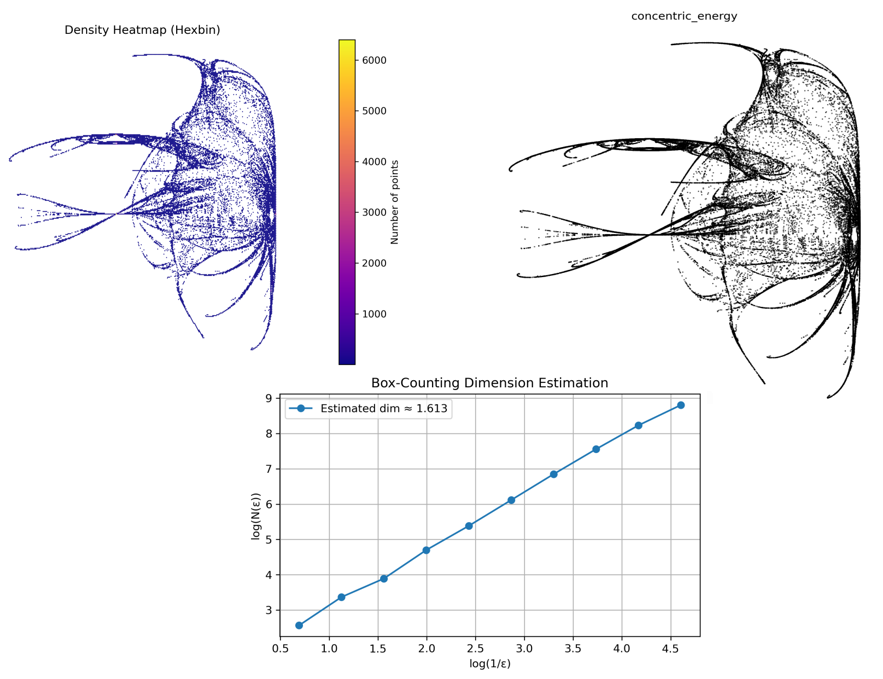}
    \caption{Experiment 3 — Concentric Energy. Top-left: hexbin density map; top-right: raw attractor plot; bottom: box-counting dimension log-log plot.}
    \label{fig:concentric_energy}
\end{figure}

\subsection{Experiment 4: Disruptive Mixture}
\label{subsec:exp_disruptive}

The fourth configuration, titled \textit{disruptive\_mixture}, was designed to probe the effects of combining nonlinear functions with inherently divergent geometrical behaviors under an unbalanced probability structure. The selected function set consists of \texttt{f6}, \texttt{f9}, and \texttt{f10}, assigned probabilities $(0.6,\; 0.2,\; 0.2)$ respectively. This allocation places strong emphasis on \texttt{f6}, a function known for its tendency to produce diagonal stretching and local instability, while the remaining functions contribute oscillatory and curvature effects in weaker proportions.

The resulting attractor (Figure~\ref{fig:disruptive_mixture}) showcases a fragmented structure dominated by sharp bifurcations and disconnected filaments. The attractor lacks any form of radial or axial symmetry and appears to evolve in an unpredictable yet bounded region. In some segments, the points seem to align along smooth arcs, while in others, abrupt directional shifts and voids disrupt visual continuity. This duality reflects the internal tension between stabilizing and destabilizing function components.

The density heatmap further illustrates this fragmentation. High-density cores are interspersed with regions of near-emptiness, suggesting that some orbits repeatedly fall into specific traps induced by the dominant function, while others scatter rapidly under the influence of the secondary maps. Notably, the interaction between \texttt{f9} and \texttt{f10} introduces local curvature and symmetry, but their lower activation frequency prevents the emergence of coherent macro-structures.

From a dimensional perspective, the box-counting estimate yields $\dim_B \approx 1.432$, reflecting moderate geometric complexity. The attractor clearly surpasses a linear curve in spatial spread, but the gaps and irregular voids reduce its capacity to occupy space efficiently. The log-log plot confirms the scaling behavior, albeit with minor deviations in the lower-resolution regime—possibly a consequence of the attractor’s discontinuous substructure.

This experiment demonstrates how probabilistic dominance by a geometrically “disruptive” function can suppress the stabilizing influence of smoother maps, resulting in a structurally incoherent yet fractal attractor. Such configurations are particularly relevant when modeling systems characterized by intermittent chaos or asymmetrical phase behavior, where no single geometric archetype prevails but multiple incompatible tendencies coexist.

\begin{figure}[H]
    \centering
    \includegraphics[width=\linewidth]{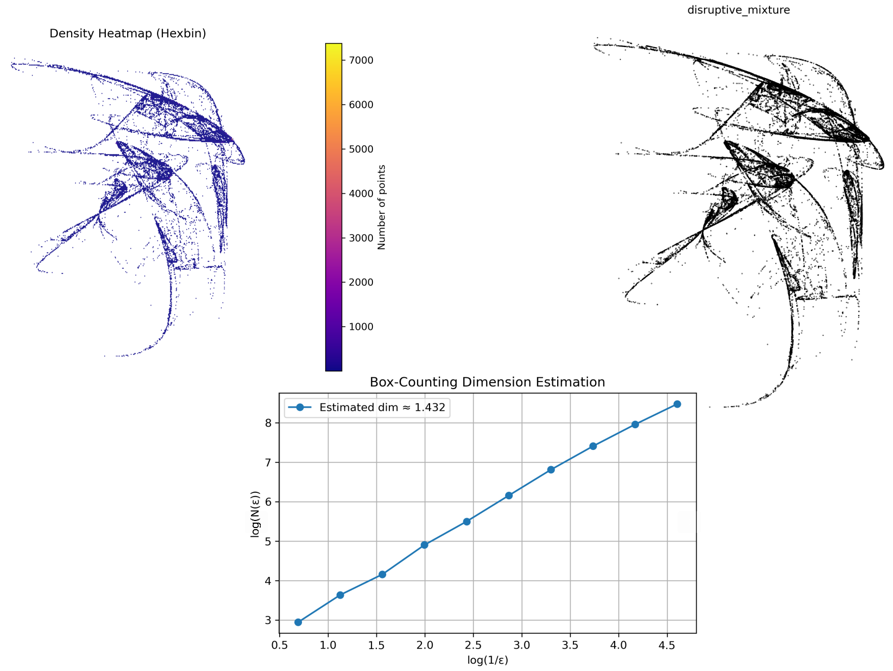}
    \caption{Experiment 4 — Disruptive Mixture. Top-left: hexbin density map; top-right: raw attractor plot; bottom: box-counting dimension log-log plot.}
    \label{fig:disruptive_mixture}
\end{figure}

\subsection{Experiment 5: High-Frequency Disturbance}
\label{subsec:exp_highfreq}

The fifth experiment, labeled \textit{high\_freq\_disturbance}, was crafted to investigate the emergent geometry produced by maps with strong oscillatory and frequency sensitive components. It uses only two functions: \texttt{f11} and \texttt{f12}, with equal probabilities $(0.5,\; 0.5)$. Despite the apparent simplicity in size, the selected functions exhibit high-frequency behavior due to nested trigonometric and hyperbolic terms specifically, $\sin(3x)$ and $\tanh(x + y)$ in \texttt{f11}, and multiplicative sinusoidal interactions in \texttt{f12}.

The resulting attractor (Figure~\ref{fig:high_freq_disturbance}) is a remarkable manifestation of structured turbulence. At a glance, the geometry appears bilaterally symmetric, with two dominant lobes extending outward from a tightly clustered central core. These lobes resemble magnetic field lines or unfolded petals, each populated with rich micro-structures and sharp directional shifts. The attractor seems to pulsate outward in layered waves likely caused by resonance-like reinforcement of the underlying periodic components.

The hexbin density plot reveals localized “bursting zones,” where density rapidly peaks along narrow paths. These concentrated ridges are juxtaposed against broad, low-density halos, suggesting that orbits tend to lock into precise rhythmic corridors before escaping into more chaotic transients. The central region, in particular, acts as both an origin and a bottleneck, from which trajectories are repeatedly launched with rotational asymmetry.

With an estimated box-counting dimension of $\dim_B \approx 1.682$, this attractor is among the densest in the set of experiments. Its dimension approaches that of a space-filling curve, reflecting both the high spatial occupation and the layered texture observed visually. The log-log regression confirms fractal scaling with minimal deviation across multiple scales, indicating stable complexity in the presence of oscillatory volatility.

This experiment powerfully illustrates how RNIFS with a small number of high-frequency nonlinear maps can generate intricate, directionally modulated attractors. The fine-grained texture and strong central pulse resonate with real-world systems exhibiting wave interference, signal feedback, or chaotic resonance phenomena. The interplay between structure and disorder here exemplifies the rich design space available in frequency-driven RNIFS dynamics.

\begin{figure}[H]
    \centering
    \includegraphics[width=\linewidth]{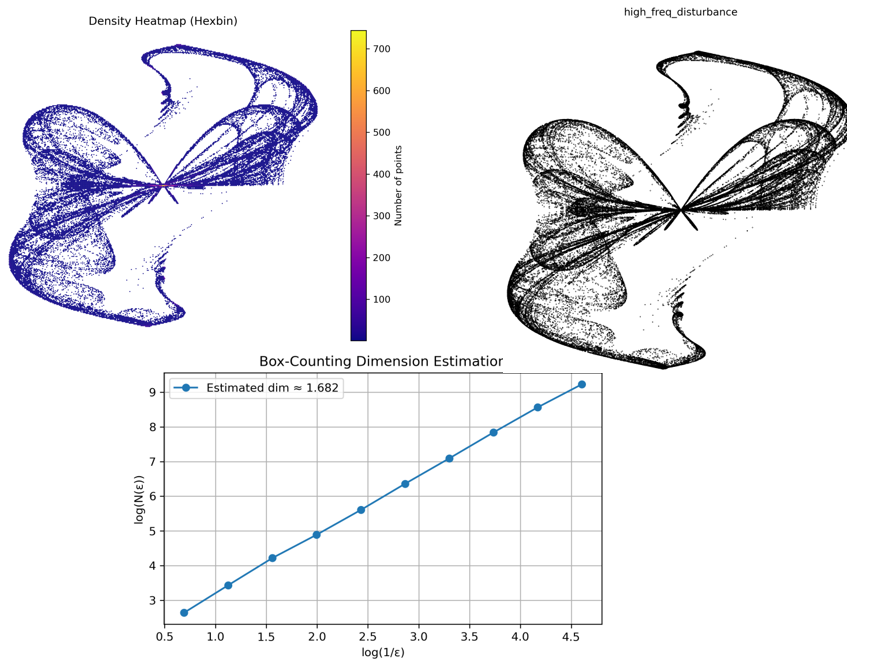}
    \caption{Experiment 5 — High-Frequency Disturbance. Top-left: hexbin density map; top-right: raw attractor plot; bottom: box-counting dimension log-log plot.}
    \label{fig:high_freq_disturbance}
\end{figure}

\subsection{Experiment 6: Spiral Rotation}
\label{subsec:exp_spiral}

The sixth experiment, named \textit{spiral\_rotation}, was designed to investigate how rotational dynamics and sinusoidal function interactions can produce densely packed, spiral like attractors. The system utilizes three functions \texttt{f3}, \texttt{f7}, and \texttt{f11} each incorporating sine or hyperbolic sine terms with balanced selection probabilities $(0.4,\; 0.3,\; 0.3)$. These functions collectively introduce both angular displacement and nonlinear modulation across the $x$ and $y$ axes.

The resulting attractor (Figure~\ref{fig:spiral_rotation}) displays a stunningly cohesive structure, characterized by radial curvature and intricate internal spiraling. Unlike fragmented or outward-dispersed attractors, this configuration remains largely enclosed, forming a bounded domain filled with swirls and rotational paths. The attractor appears almost like a vortex: points tend to spiral inward, accumulate near the center, then escape tangentially only to be reabsorbed into new trajectories.

The density plot confirms the attractor's exceptional coherence. Unlike other configurations with fragmented or sparse outskirts, this system maintains relatively uniform density throughout its domain. High-concentration areas appear as tightly wound spirals, suggesting resonance or constructive interference among function compositions. This balanced dynamic contributes to the attractor’s visual stability and high point occupancy.

Most notably, this experiment yielded the highest estimated fractal dimension among all tested configurations: $\dim_B \approx 1.892$. This value approaches the theoretical upper bound for a two-dimensional fractal set without achieving full area filling. The log-log box-counting regression reveals a near-perfect linear relationship, with minimal error across scales—an indication of scale-invariance and deep geometric nesting within the attractor.

Overall, this experiment underscores the power of carefully selected trigonometric maps to produce richly structured, highly stable RNIFS attractors. The emergent spiral geometry, in particular, mirrors natural phenomena such as fluid vortices or rotating galaxies demonstrating the model’s potential for simulating organized, self-sustained complexity arising from simple stochastic rules.

\begin{figure}[H]
    \centering
    \includegraphics[width=\linewidth]{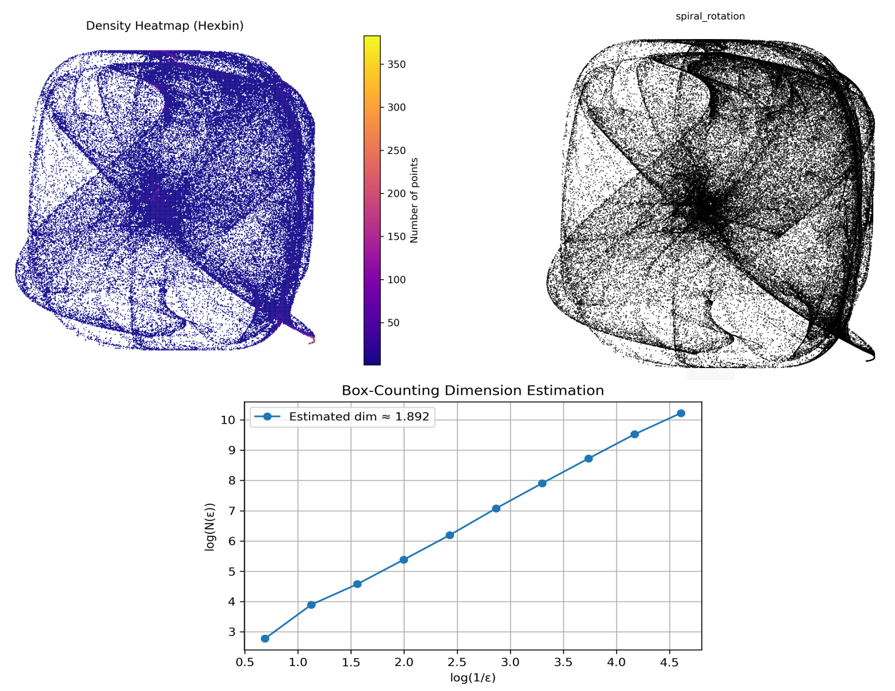}
    \caption{Experiment 6 — Spiral Rotation. Top-left: hexbin density map; top-right: raw attractor plot; bottom: box-counting dimension log-log plot.}
    \label{fig:spiral_rotation}
\end{figure}

\subsection{Experiment 7: Ultra-Resolution Analysis}
\label{subsec:exp_ultrares}

The seventh experiment, labeled \textit{ultra\_res\_analysis}, was designed with a singular goal: to probe the internal structure of RNIFS generated attractors with high numerical precision. Unlike previous simulations, this configuration was executed with a significantly larger number of points ($300{,}000$) and a prolonged burn in phase ($5{,}000$ iterations). The function set combines \texttt{f4}, \texttt{f5}, and \texttt{f8} each contributing nonlinear curvature and geometric deformation with a balanced probability distribution $(0.3,\; 0.4,\; 0.3)$.

The resulting attractor (Figure~\ref{fig:ultra_res_analysis}) is a dense and intricately layered structure, exhibiting multiple “shells” or concentric enclosures that interact non-trivially. While many RNIFS systems yield fragmented or branching patterns, this configuration stabilizes into a highly compact domain that reveals new geometrical sublayers as resolution increases. The shape resembles a network of partially inflated surfaces, each enclosing subregions of high recurrence.

The density plot provides further insight into this multi-scale organization. Bright ridges highlight orbital basins zones where trajectories orbit repeatedly due to function composition resonance. Meanwhile, faint lines and cavities reveal subtle disjunctions between zones, often caused by small numerical instabilities or rare function transitions. These micro-voids are not errors but emergent features of the RNIFS geometry, only visible at ultra-resolution.

The box-counting dimension was estimated at $\dim_B \approx 1.751$, a relatively high value confirming the attractor's space-filling tendency without complete area saturation. The log-log plot demonstrates excellent linearity, validating the presence of consistent self-similarity across several scales. Notably, the regression slope remains stable even at fine $\varepsilon$ values, suggesting that the attractor maintains structural depth at micro-levels rarely captured in lower-resolution simulations.

This experiment highlights the potential of RNIFS systems to generate attractors with deep, self-nested complexity when run under high-precision regimes. It affirms that the observed “fractality” is not superficial, but reflects robust multi-scale organization intrinsic to the interplay between nonlinear maps and probabilistic selection. Ultra-resolution analysis thus emerges as a crucial tool in uncovering latent geometrical behavior that may otherwise remain hidden.

\begin{figure}[H]
    \centering
    \includegraphics[width=\linewidth]{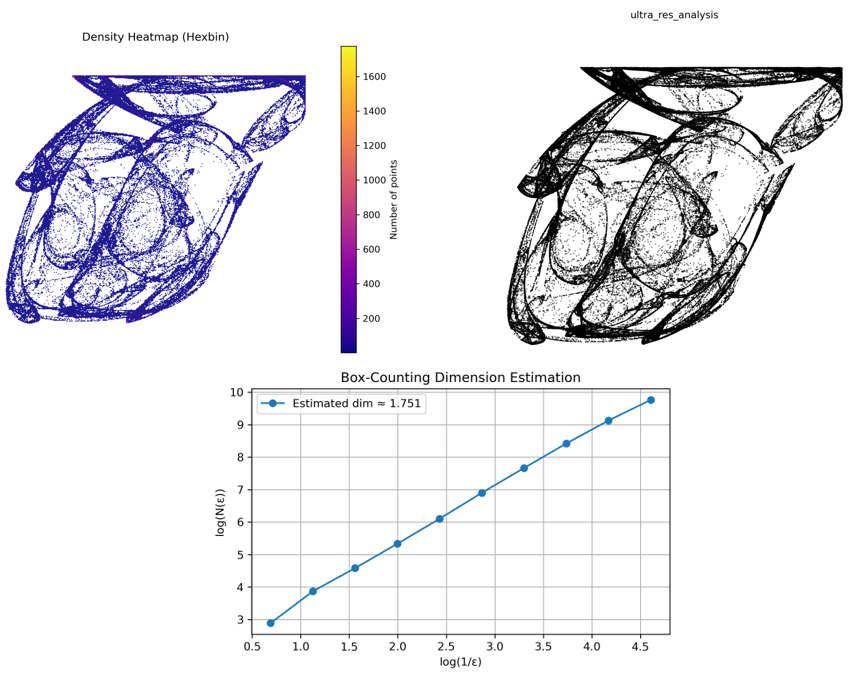}
    \caption{Experiment 7 — Ultra-Resolution Analysis. Top-left: hexbin density map; top-right: raw attractor plot; bottom: box-counting dimension log-log plot.}
    \label{fig:ultra_res_analysis}
\end{figure}

\subsection{Experiment 8: Webbed Structure}
\label{subsec:exp_webbed}

The final experiment in the series, titled \textit{webbed\_structure}, explores the emergence of fine-grained internal complexity through a dense, interlaced attractor. The system combines four nonlinear functions \texttt{f3}, \texttt{f5}, \texttt{f7}, and \texttt{f8} with equal selection probabilities $(0.25,\; 0.25,\; 0.25,\; 0.25)$. Each function contributes different oscillatory or folding behaviors, and their balanced distribution encourages uniform mixing rather than geometric dominance.

The resulting attractor (Figure~\ref{fig:webbed_structure}) displays a striking “webbed” configuration: a compact, almost quadrilateral boundary filled with numerous curved lines, loops, and overlapping sheets. The geometry is reminiscent of fibrous tissue or dynamic streamlines suspended in space. Unlike radial or spiraling attractors, the structure here is horizontally and vertically interlaced, giving it a layered, almost fabric-like appearance.

The hexbin density plot reveals an attractor rich in thin, high-density strands localized “threads” where orbits repeatedly collapse and fold. These threads intersect at multiple scales, generating a tapestry of overlapping zones, yet without the chaos observed in more volatile experiments. The attractor exhibits coherence without symmetry, and internal order without global regularity hallmarks of structured stochasticity.

Quantitatively, the box-counting dimension was estimated at $\dim_B \approx 1.769$, placing this attractor among the densest and most geometrically active in the study. The log-log regression is highly linear, confirming robust scaling behavior across scales. Notably, this value aligns with visual intuition: the attractor approaches space-filling behavior through repetition and crossing, rather than area saturation or radial diffusion.

This experiment highlights how RNIFS systems can give rise to internally connected, multi-threaded geometries when function combinations are chosen to complement rather than compete. The outcome is a self-organized system that exhibits intricate detail and statistical self-similarity, even in the absence of a dominant geometric signature. Such attractors may hold value in modeling interwoven phenomena ranging from network traffic flows to biological fiber systems.

\begin{figure}[H]
    \centering
    \includegraphics[width=\linewidth]{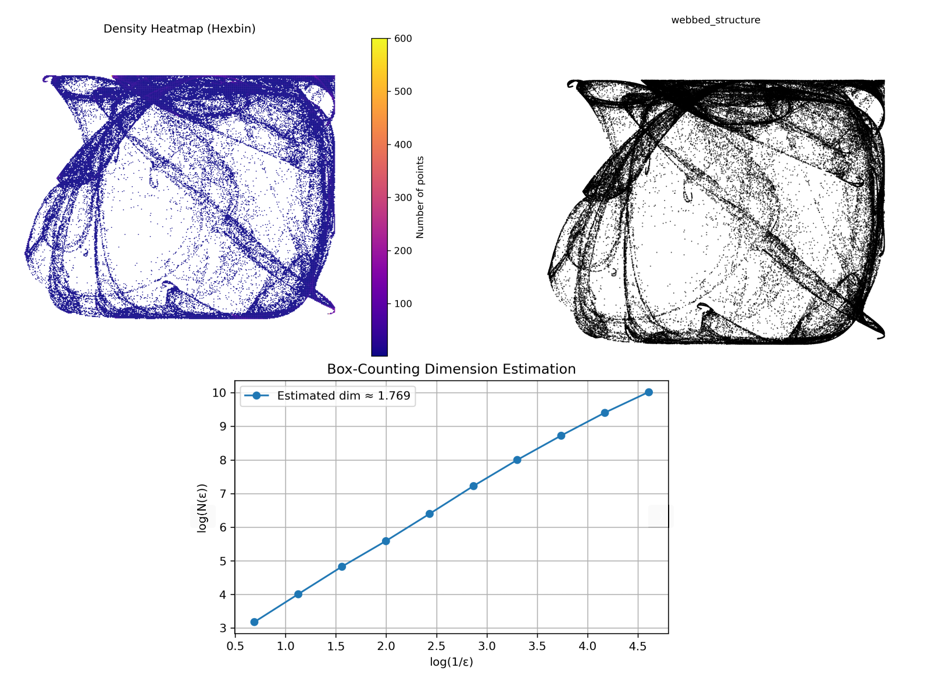}
    \caption{Experiment 8 — Webbed Structure. Top-left: hexbin density map; top-right: raw attractor plot; bottom: box-counting dimension log-log plot.}
    \label{fig:webbed_structure}
\end{figure}

\section{Case Study}
\label{sec:case_study}
\vspace{0.8em}

In this section, we examine how IFS can be viewed as a special case within the broader RNIFS framework. By doing so, we aim to validate the generality of RNIFS and demonstrate its capacity to reproduce, extend, and distort well-known fractal attractors through controlled nonlinear stochastic extensions.

\subsection{Classical IFS: The Sierpiński Triangle}
The Sierpiński triangle is a canonical example of a self-similar fractal constructed via deterministic IFS. It consists of three contractive affine maps on $\mathbb{R}^2$, each scaling the unit triangle by a factor of $1/2$ and shifting it to a distinct corner. Formally, the function family is:

\[
f_1(x, y) = \left( \tfrac{1}{2}x,\; \tfrac{1}{2}y \right),  \]
\[
f_2(x, y) = \left( \tfrac{1}{2}x + \tfrac{1}{2},\; \tfrac{1}{2}y \right), \quad \]
\[
f_3(x, y) = \left( \tfrac{1}{2}x + \tfrac{1}{4},\; \tfrac{1}{2}y + \tfrac{\sqrt{3}}{4} \right).
\]

Each function is applied with equal probability: $p_i = 1/3$. This deterministic IFS satisfies the contractivity and open set conditions, and its attractor is the well-known Sierpiński triangle, with Hausdorff dimension:

\[
\dim_H = \frac{\log 3}{\log 2} \approx 1.5849.
\]

\subsection{RNIFS Extension with Nonlinear Function}
To transition this model into the RNIFS framework, we add a nonlinear transformation defined as:

\[
f_4(x, y) = \left( \sin(\pi x) \cdot y,\; \cos(\pi y) \cdot x \right),
\] and redefine the function family $\mathcal{F}' = \{ f_1, f_2, f_3, f_4 \}$ with uniform probabilities $p_i = 1/4$ for all $i$. The inclusion of $f_4$ introduces strong nonlinearity and oscillatory behavior, violating the affine constraint of classical IFS.

This RNIFS extension retains the basic spatial structure of the original triangle, but adds significant internal distortion and stochastic fluctuation, resulting in a richer and less symmetric attractor.

From a theoretical standpoint, this example illustrates the following:

\begin{itemize}
    \item When all functions in an RNIFS are affine and fixed, the model collapses into a classical IFS.
    \item The inclusion of a single nonlinear function is sufficient to alter the geometry and dynamics of the attractor.
    \item The RNIFS retains a statistically stable attractor (supported by the invariant measure theory from Section~\ref{sec:existence_stability}), even when linearity and strict contractivity are relaxed.
\end{itemize}

Hence, classical IFS is not displaced by RNIFS, but rather embedded within it as a special case. The general RNIFS model offers a smooth continuum from well-understood deterministic structures to novel, richly structured, nonlinear random fractals.

\subsection{Experimental Comparison}
\label{sec:experimental_comparison}

To empirically validate the theoretical insights from the previous section, we conducted two experiments: one using the classical IFS representation of the Sierpiński triangle, and another using its RNIFS extension with the nonlinear transformation $f_4$. Both systems were simulated using Monte Carlo methods with $10^5$ iterations, discarding the first 100 points as burn-in.

Figure~\ref{fig:ifs_vs_rnifs} presents a side-by-side comparison of the classical and RNIFS-generated attractors. The left panel shows the canonical Sierpiński triangle with its perfectly symmetric and recursive voids. The right panel illustrates the RNIFS attractor obtained by introducing $f_4$, a nonlinear map with trigonometric dependencies.

\begin{figure}[H]
    \centering
     \includegraphics[width=\linewidth]{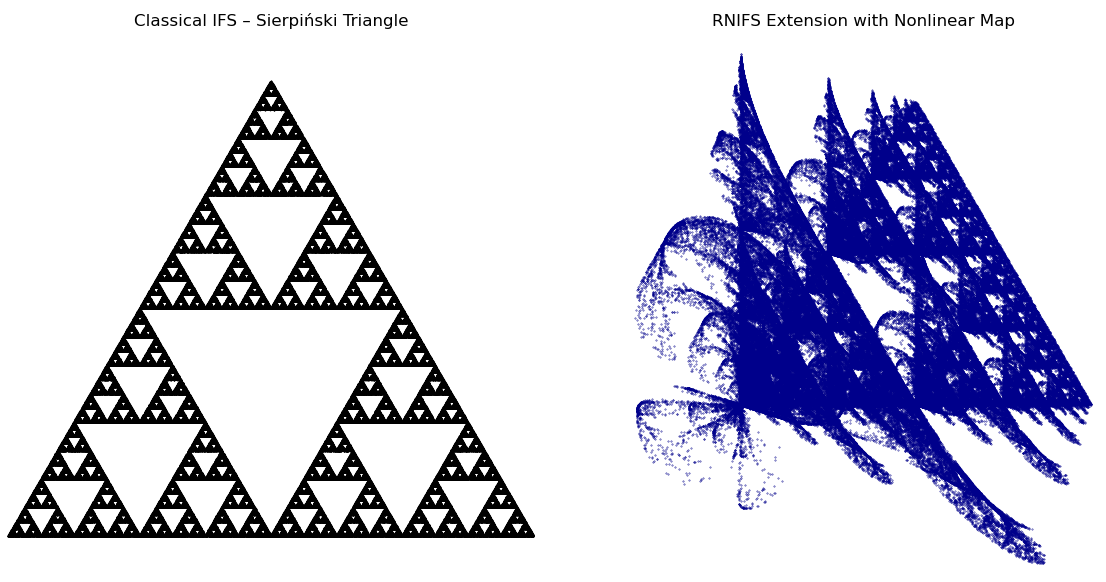}
    \caption{Left: Classical IFS (Sierpiński triangle). Right: RNIFS extension with nonlinear transformation.}
    \label{fig:ifs_vs_rnifs}
\end{figure}

We observe that the overall triangular geometry is preserved in the RNIFS case, but the internal structure becomes more chaotic, less symmetric, and more spatially complex. The attractor exhibits irregular clustering, curvature, and folding all stemming from the periodic and nonlinear nature of $f_4$. Despite these distortions, the attractor remains bounded and statistically stable, consistent with the theoretical guarantees of RNIFS dynamics.

\subsection{Fractal Dimension Estimate}

To quantify the geometric complexity of the RNIFS-generated attractor, we computed the box-counting dimension using a grid-based covering method at multiple scales. The result is plotted in Figure~\ref{fig:rnifs_boxdim}.

\begin{figure}[H]
    \centering
    \includegraphics[width=\linewidth]{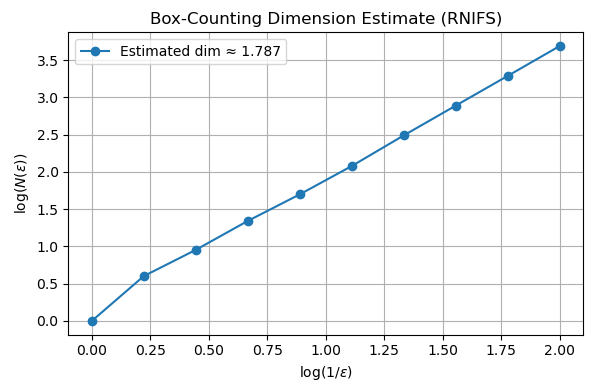}
    \caption{Box-counting dimension estimation for the RNIFS attractor. Estimated dimension: $\dim_B \approx 1.787$.}
    \label{fig:rnifs_boxdim}
\end{figure}

The dimension significantly exceeds that of the classical Sierpiński triangle ($\dim_H \approx 1.585$), confirming that the RNIFS configuration generates an attractor with higher space-filling capacity and structural richness. The log-log plot demonstrates strong linearity across multiple scales, reinforcing the presence of consistent fractal behavior despite the system’s stochasticity.

To this end, comparative experiment illustrates the power of RNIFS to extend classical IFS models both theoretically and empirically. While preserving the global shape of the original attractor, the RNIFS introduces localized distortion and structural diversity. The increased fractal dimension corroborates the enhanced complexity, offering a compelling argument for the broader applicability of RNIFS in modeling real-world phenomena beyond traditional deterministic frameworks.

\section{Conclusion}
\label{sec:conclusion}
\vspace{0.8em}

This paper has introduced and rigorously explored the framework of RNIFS, a powerful generalization of classical IFS models that integrates both nonlinearity and stochasticity. Through a combination of theoretical analysis and numerical experimentation, we established sufficient conditions for the existence and stability of fractal attractors generated by RNIFS. Specifically, we proved that under average contractivity and Lyapunov-like conditions, RNIFS admit unique invariant measures whose support forms statistically stable, self-similar structures.

Empirical simulations across diverse functional regimes revealed a wide spectrum of geometric behaviors, ranging from organized spirals and radial layers to chaotic plumes and web-like meshes. These attractors were quantitatively characterized using box-counting dimension estimates, which confirmed their fractal nature and highlighted the influence of function composition and probability distribution on spatial complexity. Notably, several attractors exhibited dimensions approaching 1.9, indicating near space-filling richness.

A comparative case study with the classical Sierpiński triangle further demonstrated that RNIFS can not only replicate known fractals as special cases but also extend them into richer, more irregular geometries through minimal nonlinear perturbations. This underscores the expressive power and generality of the RNIFS model.

Overall, our findings provide a strong foundation for further exploration of RNIFS in modeling complex, irregular phenomena across scientific domains ranging from biology and physics to cryptography and machine learning. Future work may focus on multi-scale entropy analysis, parameter sensitivity, and real-world applications where stochastic nonlinearity governs the system's evolution.

% ------------------------------------------------------------
% Declarations
% ------------------------------------------------------------
\section*{Declarations}
\vspace{0.8em}
\begin{itemize}
  \item \textbf{Funding:} Not applicable.
  
  \item \textbf{Conflict of Interest:} The authors declare that there is no conflict of interest.

\item \textbf{Availability of Data and Materials: } 
This study did not use any external datasets. All experimental results were generated through custom simulations as described in the manuscript. The full simulation code, along with configuration files and analysis scripts, is publicly available at:  
\href{https://github.com/drbouke/Mathematics/tree/main}{\texttt{https://github.com/drbouke/Mathematics}}.

  \begin{comment}
    
  \item \textbf{Author Contributions:} 
  \begin{itemize}
    \item \textbf{Mohamed Aly Bouke:} Conceptualization, Methodology, Validation, Formal analysis, Data curation, Writing – original draft, Writing – review \& editing, Visualization.
    \item \textbf{Azizol Abdullah:} Formal analysis, Supervision.
    \item \textbf{Nur Izura Udzir and Normalia Samian:} Supervision.
  \end{itemize}
  
  \end{comment}

  \item \textbf{Consent to Publish:} All authors have reviewed and approved the final version of the manuscript and have provided consent for its publication.

  \item \textbf{Ethics Approval:} Not applicable.
\end{itemize}

% ------------------------------------------------------------
% Author Biographies
% ------------------------------------------------------------
\vspace{4.6em}

\section*{Author Biographies}
\begin{tcolorbox}[biobox]
  \begin{wrapfigure}{l}{0.35\linewidth}
    \vspace{-0.5em}
    \includegraphics[width=\linewidth]{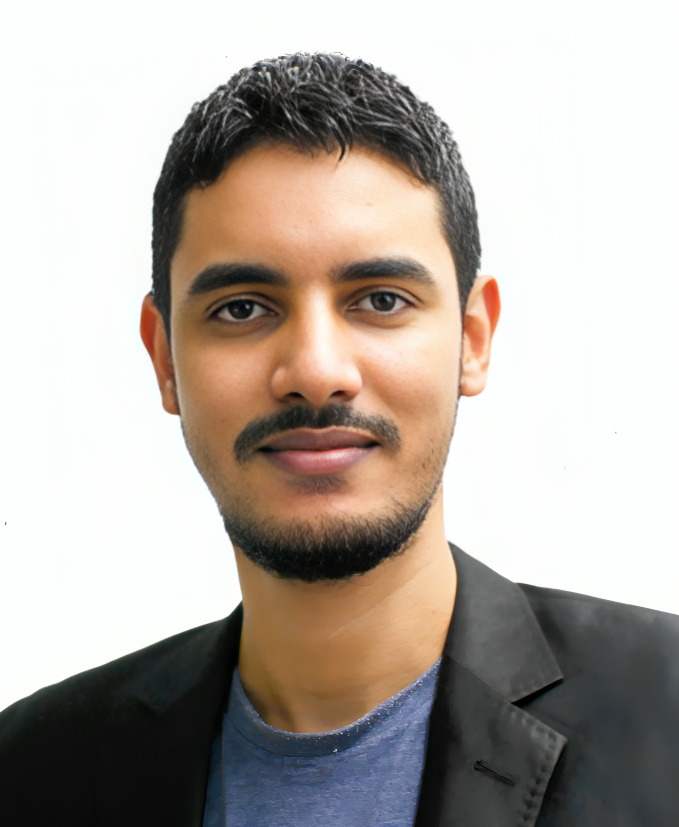}
  \end{wrapfigure}
  \textbf{Mohamed Aly Bouke}\,\href{https://orcid.org/0000-0003-3264-601X}{\includegraphics[height=1.8ex]{orcid.png}} is a researcher with interdisciplinary expertise across theoretical mathematics, computer science, artificial intelligence, and cryptography. He holds a Master’s and a Ph.D. in Information Security from Universiti Putra Malaysia and has a background in mathematics education. His academic work spans topics such as mathematical modeling, epistemic systems, AI architectures, and secure computation. Dr. Bouke is an active member of the \textit{Institute of Electrical and Electronics Engineers (IEEE)}, the \textit{International Information System Security Certification Consortium (ISC2)}, and the \textit{Institute for Systems and Technologies of Information, Control and Communication (INSTICC)}. His contributions include peer-reviewed publications, invited talks, and academic training in both technical and theoretical domains.

  \vspace{0.8em}
  \textit{Email: bouke@ieee.org}
\end{tcolorbox}

% ------------------------------------------------------------
% References
% ------------------------------------------------------------
\printbibliography

\end{document}